\documentclass[12pt, reqno]{amsart}
\input xypic
\xyoption{all}

\newcommand{\calc}{{\mathcal{C}}}

\newcommand{\caln}{{\mathcal N}}

\newcommand{\calo}{{\mathcal O}}
\newcommand{\calh}{{\mathcal H}}

\newcommand{\cali}{{\mathcal J}}

\newcommand{\cre}{{\rm Cr}}
\newcommand{\cren}{{\rm Cr}(\PP^n)}

\newcommand{\bas}{{\rm Base}}

\newcommand{\hir}{{\mathbb F}}

\newcommand{\pic}{{\rm Pic}}

\newcommand{\nef}{\overline{\rm NE}}
\newcommand{\spe}{{\rm Spec}(\C)}
\newcommand{\bideg}{{\rm bideg}}
\newcommand{\cyc}{{\rm Cyc}}

\newcommand{\cohom}{{\rm {\rm h}}}

\newcommand{\proj}{{\mathbb P}}

\newcommand{\plane}{{\mathbb P}^2}
\newcommand{\plan}{{\mathbb P}^2}
\newcommand{\spa}{{\mathbb P}^3}

\newcommand{\PP}{{\mathbb P}}
\renewcommand{\P}{{\mathbb P}}

\newcommand{\C}{{\mathbb C}}

\newcommand{\Q}{{\mathbb Q}}

\newcommand{\tor}{\xymatrix{\ar@{-->}[r]&}}
\renewcommand{\sec}{{\rm Sec}}

\begin{document}

\title[ ]{On Cremona Transformations of $\PP^3$ which factorize in a minimal form}

\newtheorem{thm}{Theorem}
\newtheorem{pro}[thm]{Proposition}
\newtheorem{cor}[thm]{Corollary}
\newtheorem{rems}[thm]{Remarks}
\newtheorem{lem}[thm]{Lemma}
\newtheorem{defi}{Definition}

\theoremstyle{remark}
\newtheorem{exa}[thm]{Example}
\newtheorem{exas}[thm]{Examples}

\newtheorem{rem}[thm]{Remark}

\begin{abstract}
We consider  Cremona transformations of the complex projective space of dimension 3 which factorize as a product of at most two elementary links of type II, without small contractions, connecting two Fano 3-folds. We show that there are essentially eight classes of such transformations and we give a geometric description of elements in each of these classes.
\end{abstract}
\author{Ivan Pan}\footnote{Partially supported by the \emph{Agencia Nacional de Investigadores} of Uruguay}
\address{Ivan Pan, Centro de Matemática, Facultad de Ciencias, Universidad de la Rep\'ublica, Igu\'a 4225, 11400 - Montevideo - URUGUAY}
\email{ivan@cmat.edu.uy}
\maketitle

\section{Introduction}\label{sec1}
Denote by $\PP^n=\P^n_\C$ the projective space of dimension $n$ over the field $\C$ of complex numbers. A Cremona transformation of $\PP^n$ is a birational map $\varphi:\PP^n\tor\PP^n$; the set $\cren$ of Cremona transformations is the Cremona group of $\PP^n$. Since the beginning, from the Luigi Cremona works in the 1860's until now-days, the Cremona transformations in the plane have been object of interest and the knowledge about this group and its elements is satisfactory (for a survey see \cite{Hu} or \cite{Alb}; for some recent works see for example \cite{Gi1}, \cite{BaBe}, \cite{Fer}, \cite{Be}, \cite{Bl2}, \cite{DoIs}, \cite{BlPaVu}, \cite{CaLa}, \cite{CeDe}, \cite{Ca12}). For higher dimension there has also been a lot of contemporary research on the subject (see for example \cite{ESB}, \cite{CK1},\cite{Pa2},\cite{RS}, \cite{PR}, \cite{GSP}, \cite{Bl11}),  though the results obtained remain sporadic and, in general, there are no substantial advances with respect to the pioneering works about either the structure of arbitrary Cremona transformations or the structure of the group $\cren$, even for $n=3$.\smallskip

On one hand, in \cite{Co} A. Corti proved that the so-called Sarkisov program, consisting of an algorithm to decompose a birational map between two Mori fiber spaces as product of elementary \emph{links}, works in dimension 3 in the context of terminal singularities. This deep result about the structure of birational maps does not give straightforward information about either the geometric structure of Cremona transformations or the structure of $\cre(\PP^3)$. Indeed, consider  the directed graph whose vertices and edges are, respectively, the (isomorphism classes of) Mori fiber spaces and the elementary links connecting two Mori fiber spaces. The Sarkisov program establish a strategy which to a Cremona transformation associates a non trivial path in that graph which starts and ends at $\PP^3$; however, that path is in general not unique and the strategy above introduces singular varieties. 
\smallskip

On the other hand, in \cite{Ka87} S. Katz rediscovers the interesting classical result that general cubo-cubic Cremona transformations (i.e. defined by the maximal minors of a $4\times 3$-matrix of general linear forms) are the unique birational maps $\varphi:\PP^3\tor\PP^3$ not defined exactly along a smooth curve. In particular he shows that such a $\varphi$ may be realized by blowing up that curve and then contracting an irreducible divisor onto a smooth curve; in other words, $\varphi$ is a special kind of elementary link of type II in the sense of the Sarkisov program (see \cite{Co}, \cite[Chap. 14]{Mat} or \cite[Chap. 5]{Do-notes}).

Then, the main result in \cite{Ka87} may be reinterpreted, in the context of the Sarkisov program, by saying that general cubo-cubic Cremona transformations realize the shortest non trivial way to go from $\P^3$ to itself in the graph mentioned above.  

Finally, in order to search other interesting Cremona transformations, and  taking into account the considerations above, it is natural to study birational maps of $\PP^3$ factorizing as a product of elementary links of type II which are ``special'' (i.e. in the smooth varieties category and without small contractions). 

In this paper we consider that problem in the following simplest case: the number of special links is $\leq 2$. We think this special case merits to be studied for several reasons. Indeed, first of all, the 3-folds associated to the vertices above (the Fano 3-folds) have historically revealed very interesting properties from which one can expect the same is true for the Cremona transformations related to them; on the other hand, this kind of transformation is extremely rare since, as we will see, it only arises for very particular degrees; finally, it shows another difference between the situation we find in $\PP^3$ and that of $\P^2$, where such a transformation can not exist as follows from a classical result of Max Noether (\cite[Lemma 2.6.1]{Alb}).      

We prove that, besides special cubo-cubic transformations, which are already special links of type II, there are essentially seven families of Cremona transformations of $\PP^3$ which factorize as a product of two special links (Theorem \ref{thm2} and Corollary \ref{cor-thm2}); as expected, these transformations are very rare, they are defined by polynomials of degrees 2, 3, 4, 6 or 9. Moreover,  we give a geometric description of these transformations when the degree 9 above does not occur (Theorems \ref{thm3} and \ref{thm4}).

In order to obtain our results we describe all links of type II, under the above specialness assumption, connecting $\PP^3$ to a projective 3-fold of Picard number 1, that is, to a rational Fano 3-fold (Theorem \ref{thm1}). We use the classification of Fano 3-folds of Picard number 1 obtained by V. Iskovskih in \cite{Isk77} and \cite{Isk78}; these papers will be our main references for Fano 3-folds properties. 

We notice that some time after this work has been archived as preprint on arxiv.org (see \cite{Pa2011}) J. Blanc and S. Lamy presented there a classification of Sarkisov links coming from the blowup of $\PP^3$ along a smooth curve; their classification (which includes also links of type I) generalize our Theorem \ref{thm1} since  they consider links containing some kind of small contractions (see \cite{BlLa2011}, to appear as \cite{BlLa2012}).    

\smallskip

\noindent{\bf Acknowledgments} We thank Arnaud Beauville and Massimiliano Mella for some useful references and comments.

\section{Special Links of type II from $\PP^3$ to a Fano 3-fold}\label{sec2} 

In this section we classify links of type II connecting $\PP^3$ to a Fano $3$-fold which are special in the sense that they arise in the category of smooth varieties and do not contain small contractions. 

Let $Z$ be a dimension 3 smooth projective variety; as always $K_Z$ denote the canonical class of $Z$. Let $p:Z\to \spa$ be a $K_Z$-negative \emph{extremal contraction} whose center is an irreducible curve $\Gamma\subset \PP^3$. By  the  Mori's classification of (divisorial) extremal contractions (\cite[Thm. 3.3 and Cor. 3.4]{Mor}) we deduce $\Gamma$ is a smooth curve and $p$ is the blow-up of this curve.

\smallskip

 The closed cone of curves $\nef(Z)$ has dimension 2 and one of its edges is the negative ray corresponding to $p$. We assume that the other edge of that cone is also a $K_Z$-negative extremal ray; denote by $q:Z\to X$ the extremal contraction associated to it, which may not be a flip since $Z$ is smooth. We fix an embedding $X\subset \proj^N$ which is supposed to be non-degenerate (the case $X=\spa$ is not excluded).

When $\dim\,X=3$ it is normal, $\Q$-factorial, and has at most terminal singularities.  A general codimension 2 linear space in $\proj^N$ intersects $X$ along a smooth irreducible curve, say  of genus $g$; we say $g$ is the \emph{sectional genus by curves} of $X$. When $\dim\,X=2$, then $X$ is a smooth surface and a general fiber of $q:Z\to X$ is a smooth rational curve. In the case where $\dim\,X=1$ this curve is smooth and a general fiber of $q$ is a del Pezzo surface (\cite[Thm. 3.5]{Mor}.

On the other hand, the rational map $\chi:=q\circ p^{-1}:\spa\tor X$ is defined by a linear system $\Lambda_\chi$ of degree $n\geq 1$ surfaces containing $\Gamma$, a subsystem of $|\chi^*\calo_{\PP^N}(1)|$; we denote by $m\geq 1$ the multiplicity of a general member of that linear system at a general point of $\Gamma$.

\smallskip

The following result was inspired by \cite{Ka87} (see also \cite[Lemma 3]{Pa11}).  

\begin{pro}
Suppose $\dim\,X>1$ and let $d:=\deg \Gamma$, $d_0:=(\dim\,X-2)\deg\,X$ and $g_0=(\dim\,X-2)g$. Then 
\begin{eqnarray}
(n^2-m^2d)(4m-n)&=&2m(d_0+1-g_0)-d_0\label{eq1.1}\\
n^2> m^2d& & \label{eq1.2}\\
m^2|(n^3-d_0)\, , & & m|(n^2(n-2)+1-g_0)\label{eq1.3}.
\end{eqnarray} 
\label{pro1.1}
\end{pro}

\begin{proof}
Take a general divisor $H\in|p^*\calo_{\spa}(1)|$ and let us denote $E:=p^{-1}(\Gamma)$. By construction, $q$ is defined by the complete linear system $|nH-mE|$. A  general element in $|nH-mE|$ corresponds by $p$ to a degree $n$ surface in $\P^3$ passing through $\Gamma$ with multiplicity $m$ from which it follows (\ref{eq1.2}). 

The projection formula implies 
\[H^3=1, H^2\cdot E=0, H\cdot E^2=-d;\]

Therefore, we have
\begin{equation}
(nH-mE)^3=n^3-3nm^2d-m^3E^3=d_0,\label{eq1.4}
\end{equation}
since it corresponds to intersect $p(Z)=X\subset \P^N$ with a general codimension 3 linear space. Then $m^2$ divides $n^3-d_0$ as stated in (\ref{eq1.3}). 

On the other hand, since $q:Z\to q(Z)=X\subset\proj^N$ is a morphism, $\dim X>1$ and $Z$ is smooth, Bertini's Theorem implies that the strict transform of a general codimension 2 linear space in $\proj^N$, under $q$, is a smooth complete intersection of two smooth members of $|nH-mE|$; of course, when $\dim X=2$ that complete intersection is the disjoint union of $\deg X$ smooth rational curves. The adjunction formula yields 
\begin{eqnarray}
2g_0-2&=&(nH-mE)^2\cdot\left[2(nH-mE)+p^*(K)+E\right]\nonumber\\
&=&(n^2H^2-2nm H\cdot E+m^2E^2)\cdot\left[(2n-4)H-(2m-1)E\right]\nonumber\\
&=&2\left[n^2(n-2)-nmd(2m-1)-m^2(n-2)d\right]-\nonumber\\
& &m^2(2m-1)E^3;\label{eq1.5}
\end{eqnarray}
in particular, we obtain the second divisibility condition in (\ref{eq1.3}).

Finally, by eliminating $(2m-1)m^3E^3$ between (\ref{eq1.5}) and (\ref{eq1.4}), we obtain the relation (\ref{eq1.1}).
\end{proof}

A dimension 3 projective variety $X$ is said to be a Fano 3-fold if it is smooth and its anti-canonical class  $-K_X$ is ample. In \cite{Isk77} and \cite{Isk78} V. Iskovskih gives an essentially complete classification of Fano 3-folds with Picard number 1 (see also \cite{Isk83}); following Iskovskih the \emph{index} of such a Fano 3-fold is the positive integer $r$ such that $-K_X\simeq\calh_X^r$, where $\calh_X$ generates the Picard group of $X$. On the other hand, we need to consider Fano 3-folds which are rational and in the present work we only pay attention to them except once (maybe twice) when we will consider a Fano 3-fold whose rationality was not yet established (according to \cite{IskSho} and \cite{Isk83}).

\smallskip

Let us recall what an elementary link of type II is. 

Let $X_1,X_2$ be normal projective dimension 3 varieties which are $\Q$-factorial and have at most terminal singularities. Suppose $X_i$ admits a Mori fiber structure, that is a fibration $X_i\to Y_i$, $\dim\,Y_i<\dim\,X_i$, which is an extremal $K_{X_i}$-negative contraction, for $i=1,2$. A \emph{link of type II} connecting $X_1$ to $X_2$ is  a birational map $\chi:X_1\tor X_2$ which admits a decomposition 
\[\xymatrix{X_1&Z_1\ar@{->}[l]_{\tau_1}\ar@{-->}[r]&Z_2\ar@{->}[r]^{\tau_2}&X_2}\] 
where $\tau_i$ is an divisorial extremal contraction with respect to $K_{Z_i}$, $Z_i$ is also a normal projective (dimension 3) variety, $\Q$-factorial and having at most terminal singularities, $i=1,2$, and $Z_1\tor Z_2$ a sequence of logarithmic flips (small contractions); in this work we will always deal with the case $Z_1=Z_2$, that is, there are no small contractions, and where $X_1=\PP^3$ and $X=X_2$ is smooth. Since $X$ has Picard number 1  the Mori fiber structure is nothing but $X\to\spe$ and we know $X$ is a Fano 3-fold, that is, its anti-canonical class $-K_X$ is ample. We say $\chi$ is nontrivial if it is no defined somewhere. \smallskip 

\begin{defi}\label{defi-special}
A link of type II is said to be special if it arises into the category of smooth varieties and does not contain small contractions.  
\end{defi} 

\begin{rem}
If $\chi=q\circ p^{-1}:\PP^3\tor X$ is a special link of type II, the center of $p$ is a curve: indeed, otherwise $p$ is the blow-up of a point (\cite[Thm. 3.3]{Mor}) and then $q$ is a Mori Fiber Space. 
\label{rem1.0}
\end{rem}

In the following series of examples we give a list of special links of type II connecting $\spa$ to a Fano rational 3-fold. For a curve $\Gamma\subset\PP^3$ we denote by $\cali_\Gamma\subset \calo_{\PP^3}$ its ideal sheaf. 

\begin{exas}\ \\

\noindent$(L.1)$. Let $\Gamma\subset\spa$ be a smooth quintic curve of genus 2. Then $h^0(\cali_\Gamma(3))=6.$ If $\phi:\spa\tor\proj^5$ is the rational map defined by the choice of a basis of $H^0(\cali_\Gamma(3))$, it is birational onto its image; this image is exactly the famous quadratic complex of lines, that is, a complete intersection of smooth hyperquadrics (see for example \cite[Chap. 6]{GH}).

\smallskip

\noindent$(L.2)$. Let $\Gamma\subset \spa$ be a smooth rational curve of degree 4. Using a parameterization $\eta:\PP^1\to \Gamma$ by polynomials of degree 4 we obtain $h^0(\cali_\Gamma(3))=7$ and $h^0(\cali_\Gamma(2))= 1$; take $S$ and $Q$ irreducible cubic and quadric containing $\Gamma$, respectively. Hence $S\cap Q=\Gamma\cup C$ where $\deg\,C=2$. If $C$ were in a plane, then $\Gamma$ would be arithmetically Cohen-Macaulay which contradicts formulae in \cite[Prop. 3.1]{PS74}, since $\Gamma$ has arithmetic genus 0. Therefore $C$ is the union of two skew lines.

We obtain a rational map $\phi:\spa\tor \proj^6$ defined by cubic homogeneous polynomials. By restriction  to a general plane we easily conclude that the blow-up $p:Z\to \PP^3$  of $\spa$ along $\Gamma$ resolves the indeterminacies of $\phi$; denote by $q:Z\to \proj^6$ the morphism obtained by composing $\phi$ with that blow-up.  If $H\subset X$ is the pullback of a general plane in $\spa$, then $q^*\calo_{\proj^6}(1)=\calo_Z(3H-E)$ where $E$ is the exceptional divisor produced by blowing up $\Gamma$. Since $p_*(E^3)=2-2g(\Gamma)+K_{\spa}\cdot \Gamma=-14$, from (\ref{eq1.4}) we obtain $(3H-E)^3=5$. Therefore $q(Z)\subset\proj^6$ is a degree 5 threefold. 

Finally, fix a point $x\in\Gamma$ and denote by $W_x$ the strict transform in $Z$ of a general plane passing through $x$; by construction the restriction of $q$ to $W_x$  is the anticanonical embedding of a del Pezzo surface of degree $5$. We conclude that $X:=q(Z)$ is smooth.

\smallskip

\noindent$(L.3)$. Let $\Gamma\subset \proj^3$ be a smooth conic. Then $h^0(\cali_\Gamma(2))=5$ and  we obtain a rational map $\phi:\spa\tor \proj^4$ defined by quadratic homogeneous polynomials. This is a link of type II obtained by blowing up $\Gamma$ and contracting the (strict transform of) chords of this conic.

\smallskip

\noindent$(L.4)$ Let $\Gamma\subset\spa$ be a smooth elliptic curve of degree 5. Then $\cohom^0(\calo_\Gamma(3))=15$. It follows that $r:=\cohom^0(\cali_\Gamma(3))\geq 5$; denote by $\phi:\spa\tor\proj^{r-1}$ the rational map defined by the cubic forms vanishing on $\Gamma$. If $p:Z\to \spa$ is the blow-up of $\spa$ along $\Gamma$, by restricting $\phi$ to planes, we note that $q=\phi\circ p$ is a morphism whose image, say $X$, is a 3-fold. Take a general divisor $H\in|p^*\calo_{\spa}(1)|$ and let us denote $E:=p^{-1}(\Gamma)$. By construction, $q$ is defined by the complete linear system $|3H-E|$. 

We have $H\cdot E^2=-5$ and $E^3=-20,$ so $(3H-E)^3=2$. Since $q(Z)$ is nondegenerate, we deduce $q$ is birational onto its image which is  a dimension 3 variety of degree 2, that is, $r-1=4$ and $Q:=q(Z)\subset \proj^4$ is a hyperquadric. 

Projecting from a point of $\Gamma$ to $\plan$, we see that through such a point, there pass two 3-secant lines of $\Gamma$; then the union of the strict transforms, in $Z$, of these 3-secant lines defines a (divisor) surface $F$. Moreover, since 
\[K_Z=-4H+E=-3\calh_Z+F,\]
where $\calh_Z\in |3H-E|$, we deduce, on one side, $q$ is an extremal contraction which contracts $F$, and on the other side, $F\in |5H-2E|$, that is, $p(F)$ is a degree 5 surface passing through $\Gamma$ with multiplicity $2$. 

Finally, we have $F^2\cdot H_Z=-5$ and $ F^3=-15$ so $q(F)$ is an elliptic quintic curve on $Q$. Moreover, this quintic is not contained in a dimension 3 linear space: indeed, otherwise the class $\calh_Z-F=-2H+E$ would be represented by an effective divisor but $H^2\cdot(-2H+E)=-7$.

\smallskip

\noindent$(L.5)$. Let $\Gamma\subset\spa$ be a smooth genus 3 curve of degree 6 which is arithmetically Cohen-Macaulay. That is, the ideal sheaf $\cali_\Gamma$ is generated by the maximal minors of a $4\times 3$ matrix of linear forms. These four homogeneous polynomials define a birational map $\chi:\spa\tor\spa$ whose inverse is defined also by cubic polynomials. This map is a link of type II obtained by blowing up the curve $\Gamma$ and then contracting the (strict transform of) the trisecant lines to $\Gamma$ (see \cite{Ka87})   
\label{exas1}
\end{exas} 

\begin{rem}
In case $(L.2)$ above the unique quadric containing $\Gamma$ is smooth: in fact, our argument above shows that quadric contains two skew lines.
\label{rem2}
\end{rem}

\begin{rem}
Note that in the case where $d_0=2g_0-2$, the equation  (\ref{eq1.1}) becomes
\begin{equation}
(n^2-m^2d)(4m-n)=(m-1)d_0;\label{eq1.6}
\end{equation}
in this case $(m,n,d)=(1,n,n^2)$ always gives a solution. On the other hand, there is no birational map $\chi:\spa\tor X$ satisfying such a solution: indeed, this solution corresponds to a pencil. 
\label{rem1}
\end{rem}

A divisorial extremal $K_Z$-negative contraction $p:Z\to Y$ is (also) called  an \emph{extraction}; denote by $E$ the exceptional divisor associated to $p$. \smallskip

Consider now a birational map $\chi:\PP^3\tor X\subset\PP^N$ and let $\calh$ be the linear system of strict transforms, under $\chi$, of hyperplane sections of $X$. Hence $\calh$ is the \emph{homaloidal transform} of $|\calo_X(1)|$; it is a linear system of surfaces, of degree $n$ say. If $p:Z\to\PP^3$ is an extraction, the homaloidal transform $\calh_Z:=p_*^{-1}\calh$ of $\calh$ satisfies
\[\calh_Z=p^*\calh-m_E E\]
for a integer $m_E\geq 0$. The ramification formula for $p$ is 
\[K_Z=p^*K_{\PP^3}+a_E E,\ \]
where $a_E>0$ is the \emph{discrepancy} of $E$ which is known to be a rational number. If $\chi$ is not a morphism, the Noether-Fano-Iskovskih Criterion (\cite[Prop. 13-1-3]{Mat}) implies $m_E/a_E>n/4$.\smallskip

In the following Lemma and Theorem we need to solve the arithmetic equation (\ref{eq1.1}) under the conditions (\ref{eq1.2}) and (\ref{eq1.3}) where $d_0$ and $g_0$ take some fixed values. For this porpose we use \emph{Maxima Algebra System}, a free computer algebra system (see details in \cite{computations}); nevertheless we try there to bound the values of solutions in such a way that the reader may, if preferred, do all computations by hand instead (maybe a long and tedious work).

\begin{lem}
Let $X$ be a rational Fano 3-fold of Picard number 1 and index $r=1$. Then there are no non trivial special links of type II connecting $\PP^3$ to $X$.
\label{lem1}
\end{lem}
\begin{proof}
Suppose there exists a non trivial special link of type II, say $\chi=q\circ p^{-1}:\spa\tor X$. We will show this gives a contradiction. 

First of all we notice that $p:Z\to\spa$ is the blow-up of an irreducible smooth curve, say $\Gamma$, of degree $d$ (Remark \ref{rem1.0}); we keep integers $m$ and $n$ as in Proposition \ref{pro1.1}. Therefore the discrepancy associated to $p$ is $a_E=1$. Moreover, since the  birational map $\chi$ is not a morphism then $n<4m$. Since $\chi$ is birational we get $m<n$ (inequality (\ref{eq1.2})); hence the equation (\ref{eq1.6}) implies $m>1$.

Second, since $X$ is rational and following the  Iskovskih Classification of Fano 3-folds (see \cite[Thm. 1]{Isk83}) we may suppose either $g_0=6$ and $X=X_{10}\subset\PP^9$ has degree 10 or $X=X_{2g_0-2}\subset\PP^{g_0+1}$ has degree $d_0=2g_0-2\in\{6, 10,12,16,18,22\};$ note that for degree $6$ and degree $10$ cases $X$ is known to be non-rational only for general varieties; however, in  \cite[Thm. 6 and Cor.]{Isk83} Iskovskih asserted that any smooth complete intersection of type $(2,3)$ in $\PP^5$ is not rational; we did not find the reference he gave for it but we omit this case in our proof since it may be worked analogously to the other cases.

In the sequel we then assume that the center of $p$ is a smooth and irreducible curve $\Gamma$ of degree $d\geq 1$. 

Take two general degree $n$ surfaces $S, S'$ in the linear system $\Lambda_\chi$ defining $\chi=q\circ p^{-1}$. Then, thinking of $S\cap S'$ as a 1-cycle we have $S\cap S'=m^2\Gamma+ T$ where $T$ is a curve of degree $n^2-m^2d$ which is the strict transform by $\chi$ of a general codimension 2 linear section of $X$; in particular $T$ has (geometric) genus $g_0$. \smallskip

a)  Case $d_0=10$ and $g_0=6$. From (\ref{eq1.3}) we obtain  $m^2|(n^3-10)$ and $m|(n^3-2n^2-5)$. Hence $m$ divides
\[135=(2n^2+4n-11)(n^3-10)-(2n^2+8n+5)(n^3-2n^2-5).\]
Moreover, $m< n<4m$ by (\ref{eq1.2}). The solutions $(m,n,d)$ for the equation (\ref{eq1.6}) are
\[(3,7,5), (3,10,10).\]
The first case is excluded for otherwise the curve $T$ above should be a quartic curve of genus 6. In the second one, suppose that solution gives a birational morphism $q:Z\to X_{10}\subset \proj^{7}$. From Equation (\ref{eq1.4}) we deduce
\begin{eqnarray*}
10&=&(10H-3E)^3\\
          &=&-1700-27E^3
\end{eqnarray*}
which implies $E^3$ is not an integer number: contradiction.
\smallskip

c) Case $d_0=12, g_0=7$.  Then $m^2|(n^3-12)$ and $m|(n^3-2n^2-6)$, and $m$ divides
\[156=(2n^2+4n-10)(n^3-12)-(2n^2+8n+6)(n^3-2n^2-6).\]
There are no solutions for the equation (\ref{eq1.6}).
\smallskip

d) Case $d_0=16, g_0=9$. Then $m$ divides
\[96=(n^2+2n-4)(n^3-16)-(n^2+4n+4)(n^3-2n^2-8).\]
The corresponding solutions for the equation (\ref{eq1.6}) are $(2,4,3), (2,6,7)$.

The case $(2,4,3)$ may be excluded: in fact, otherwise $T$ should be a quartic of genus $9$.

Now we exclude the case $(2,6,7)$: suppose that solution gives a birational morphism $q:Z\to X_{20}\subset \proj^{10}$; denote by $F$ the exceptional divisor associated to $q$. Comparing the ramification formulae for $p$ and $q$, respectively, we obtain
\[-4H+E=-6H+2E+F.\]
We deduce $p(F)\supset \Gamma$ is a quadric surface. By restricting to a general plan we get all sextic surfaces singular along $\Gamma$ must contain that quadric: contradiction. \smallskip

e) Case $d_0=18, g_0=10$. Then $m$ divides
\[414=(4n^2+8n-14)(n^3-18)-(4n^2+16n+18)(n^3-2n^2-9).\]
There are no solutions for the equation (\ref{eq1.6}) in this case.
\smallskip

f) Finally we deal with the case $d_0=22, g_0=12$. Then $m$ divides
\[464=(4n^2+8n-10)(n^3-22)-(4n^2+16n+22)(n^3-2n^2-11).\]
We obtain the solution $(7,16,5)$. This solution may also be excluded: indeed, suppose that solution gives a birational morphism $q:Z\to X_{22}\subset \proj^{13}$. Since $H\cdot E^2=-5$ we deduce
\begin{eqnarray*}
\deg\, X_{22}=22&=&(16H-7E)^3\\
          &=&-7664-7^3E^3\\
\end{eqnarray*}
which implies $E^3$ is not an integer number: contradiction.

\end{proof}

\begin{thm}
Let $X$ be a smooth Fano 3-fold with Picard number $1$ and index $r$. Let $\chi=q\circ p^{-1}:\spa\tor X$ be a non trivial special link of type II. Then, $r\geq 2$ and we have exactly one of the following statements:

a) $r=2$, $X$ is the complete intersection of two hyperquadrics in $\proj^5$ and $\chi$ is as in $(L.1)$. 

b) $r=2$, $X$ is a quintic in $\proj^6$ and $\chi$ is as in $(L.2)$. 

c) $r=3$, $X$ is a hyperquadric in $\proj^4$ and $\chi$ is as in $(L.3)$ or $(L.4)$.

d) $r=4$, $X=\proj^3$ and $\chi$ is as in $(L.5)$.
\label{thm1}
\end{thm}

\begin{proof}

From Lemma \ref{lem1} we get $r\geq 2$ and from the Iskovskih Classification of Fano 3-folds we get that the unique cases we need to consider are those as in the statement of Theorem \ref{thm1}. We keep all notations, $m, n, d$, etc, as before.
\smallskip   

a) From (\ref{eq1.1}) we know $m^2|(n^3-4)$ and $m|(n^3-2n^2)$. Therefore $m$ divides 
\[8=(n^3-4)(n^2-2)-(n^3-2n^2)(n^2+2n+2).\]
On the other hand, since $m< n<4m$ we deduce $m\leq 2$. A straightforward computation yields $(m,n,d)=(1,3,5)$ as unique solution for (\ref{eq1.1}).

Take general cubic surfaces $S,S'$ in the linear system $\Lambda_\chi$. Then $S\cap S'=\Gamma\cup T$ where $T$ is a quartic curve whose geometric genus is $g_0=1$. Since $T$ is not contained in a plane and all nondegenerate singular quartic curve is rational it follows $T$ is smooth; in particular $T$ is arithmetically Cohen-Macaulay. By \emph{liason} (\cite[Prop. 3.1(vi)]{PS74}) we deduce $\Gamma$ is a an arithmetically Cohen-Macaulay curve of arithmetic genus 2.

\smallskip
 
b) An easy computation using the adjunction formula yields $g_0=1$ in this case. From (\ref{eq1.1}) we obtain $m^2|(n^3-5)$ and $m|(n^3-2n^2)$. Hence $m$ divides 
\[15=(2n^2-3)(n^3-5)-(2n^2+4n+5)(n^3-2n^2);\]
as we know $m< n<4m$, and then $m>1$ is easily seen to be not possible. Solving equation (\ref{eq1.1}) we get the unique solution $(1,3,4)$.

Now we know $\Gamma$ is a nondegenerate smooth quartic, it suffices to show that curve is not elliptic. In fact, such an elliptic curve is a complete intersection of quadrics which implies $h^0(\cali_\Gamma(3))=8$, and then $X$ is the projection of a dimension three variety $X'\subset \proj^7$, from a point $p\not\in X'$. The birationality of $\chi$ implies that projection is an isomorphism, which is not possible because $X$ is projectively normal (see \cite[Prop. 4.4(iii)]{Isk77}).

\smallskip

c) Then $d_0=2$ and $g_0=0$. As before we have $m$ divides
\[5=(6n^2-4n-7)(n^3-2)-(6n^2+8n+9)(n^3-2n^2+1);\]
we obtain the solutions $(1,2,2)$ and $(1,3,5)$ and, as before, we see these are the required solutions. 

\smallskip

d) Now $d_0=1$ and $g_0=0$. In this case $m$ divides
\[2(n-1)=(n-2)(n^3-1)+n(n^3-2n^2+1).\]
Since $n<4m$ we need to solve (equation (\ref{eq1.1})) 
\[(n^2-m^2d)(4m-n)=4m-1,\]
where $2(n-1)=mk$ for $k=1,\ldots,7$. We get $(1,3,6)$ is the unique solution. The result follows form \cite{Ka87} in this case.
\end{proof}

\section{Associated Cremona transformations and Main Results}\label{sec3}

In this section we classify all Cremona transformations which are either a special link of type II or may be factorized as  $\chi_2^{-1}\chi_1$, where $\chi_i:\PP^3\tor X$ is a  special link of type II onto a smooth 3-fold $X\not\simeq\PP^3$ ($i=1,2$).  Since $X$ must be a Fano 3-fold, by Theorem \ref{thm1} we already know $\chi_i$ is necessarily as in examples $(L.1)$, $(L.2)$, $(L.3)$ or $(L.4)$. 

In order to establish the kind of transformation we deal with we need first to understand the geometry of links involved in  Theorem \ref{thm1}. This is the   
subject of the following paragraph.

\subsection{Geometric description of links}\label{subsec3.1}

\subsubsection{Cases  $(L.1), (L.3)$ and $(L.5)$}\label{subsubsec3.1.1}

The links of type II as in case $(L.1)$ are described in \cite[\S 4]{AGSP}. Analogously, the links of type II as in cases $(L.3)$ and  $(L.5)$ are well known (the former case is very classical and easy to describe and for the last one see \cite{Ka87}, \cite[Chap.\ VIII, \S 4]{SR} or \cite[Chap.\ HIV, \S 11]{Hu}). For the convenience of the reader we include here a slight description of the situation for these three cases:

\smallskip

\noindent{\bf Case} $(L.1)$: The Fano 3-fold $X$ is a complete intersection of two hyperquadrics in $\PP^5$. There is a dimension 2 family of lines on $X$ where  two such lines may intersect themselves or not. If $L\subset X$ is a line on $X$, we have:

(a)  there are two possibilities for the normal bundle $\caln_L X$ of $L$ in $X$: it is isomorphic to either $\calo_{\PP^1}\oplus\calo_{\PP^1}$ or $\calo_{\PP^1}(1)\oplus\calo_{\PP^1}(-1)$. 

(b) a general projection $\pi_L:X\tor \PP^3$ from $X$ with center $L$ is a birational map which defines a special link of type II connecting $X$ to $\PP^3$; this link is obtained by blowing up $L$ and then contracting the strict transform of lines on $X$ passing through a point of $L$; the set of these lines contract onto a smooth quintic curve of genus 2 which is contained in a unique quadric: it is smooth or not, depending on the two possibilities for $\caln_L X$ described above, respectively. Moreover, the inverse map $\chi:=\pi^{-1}_L:\PP^3\tor X\subset\PP^5$ is as in the first part of Examples  \ref{exas1}. 
\smallskip

\noindent{\bf Case} $(L.3)$: Here $X$ is a smooth hyperquadric in $\PP^4$. The projection $\pi_x: X\tor\PP^3$ from $X$ with center a point $x\in X$ is a birational map which defines a special link. This link is obtained by blowing up $x$ and then contracting the strict transform of lines on $X$ passing through $x$; these lines contract onto a smooth conic. Moreover, the inverse map $\chi:=\pi^{-1}_x:\PP^3\tor X\subset\PP^4$ is as in the third part of Examples \ref{exas1}.  
\smallskip

\noindent{\bf Case} $(L.5)$: The link in this case is exactly the one of  the last part of Examples \ref{exas1}. This is obtained by blowing up a smooth sextic curve of genus 3, say $\Gamma$,  which is not contained into a quadric, and then contracting the 3-secant lines to $\Gamma$ onto a curve projectively equivalent to $\Gamma$.

\subsubsection{{\bf Case} $(L.2)$}\label{subsubsec3.1.2} Let $\Gamma\subset\PP^3$ be a smooth rational  quartic curve. We first note that through a point $y\in\Gamma$ there passes a unique trisecant line to $\Gamma$: in fact, by projecting $\Gamma$ from $y$ to a general plane we get a (rational) singular cubic curve. Then the unique smooth quadric $Q$ containing $\Gamma$ (Remark \ref{rem2}) is none other than the trisecant variety $\sec_3(\Gamma)$; moreover, $Q$ is the exceptional set of the extremal contraction $q:Z\to X=X_5\subset\PP^6$ (we keep notations from \S\,\ref{sec2}) and the numerical equivalence class of the strict transform of  a trisecant line to $\Gamma$ generates the corresponding (negative) extremal ray. Therefore the center of $q:Z\to X$ is necessarily a smooth rational curve, say $C$ (\cite[Thm. 3.3 and Cor. 3.4]{Mor}; moreover, $C$ is a conic in this case, because the image $q(Q)$ spans a dimension 2 linear space in $\PP^6$.

Take a general plane $\Pi\subset \PP^3$; all trisecant lines to $\Gamma$ intersect $\Pi$ at a point and $\Pi\cap Q$ is a smooth conic $C'$. The strict transform $p_*^{-1}(\Pi)$ of $\Pi$ by $p:Z\to\PP^3$ is an abstract del Pezzo surface $S$ of degree 5. Since $p_*^{-1}(\Pi)=p^*(\Pi)$ and $(3H-E)^2\cdot p^*(\Pi)=5$ we deduce that $q(S)$ is a (isomorphic to $S$) hyperplane section of $X$ and $q(C')=C$. We conclude that the inverse link $\chi^{-1}:X\tor \spa$ is the restriction to $X$ of a (linear) projection of $\PP^6$ from the plane containing $C$. 

Moreover, let $\ell\subset\PP^3$ be a general line; we may suppose $\ell$ to be contained in the general plane $\Pi$ considered above. Then, the strict transform $\gamma:=\chi_*(\ell)$ of $\ell$ is a twisted cubic which is 2-secant to $C$.

Conversely, we show that all conics on $X$ are as above. Let $C\in \calc_2$ be a (smooth) conic and denote by $P\subset \PP^6$ the plane containing it. We know $X$ does not contain planes (\cite[Prop. 5.3]{Isk78}), so $P\not\subset X$. Since $X$ is cut out by hyperquadrics (\cite[Thm. 4.2]{Isk77}) we infer $X\cap P=C$ and that  the points of $\PP(\caln_C P)$ corresponding to normal directions of $C$ in $P$ do not correspond to normal directions of $C$ in $X$. If $\sigma:Z\to X$ is the blow-up of $X$ along $C$ and $\pi:X\tor \PP^3$ is the restriction to $X$ of a general projection of $\PP^6$ from $P$, then we deduce a commutative diagram
\begin{equation}
\xymatrix{&Z\ar@{->}[ld]_\tau\ar@{->}[rd]^\sigma&\\
\PP^3& &X\ar@{-->}[ll]_\pi}\label{diagram3.1}
\end{equation}
where $\tau$ is a morphism. Then, a general hyperplane section $S$ of $X$ containing $C$ is smooth, that is, it is a del Pezzo surface of degree 5. So we may realize $S$ via the blow-up  $f:S\to\PP^2$ of $\PP^2$ at points $p_1,p_2,p_3,p_4$. We conclude $C$ is realized as the strict transform of (see Lemma \ref{lem3}):

(a)  a line passing through only one of $p_1,p_2,p_3,p_4$, or

(b) a conic passing through $p_1,p_2,p_3,p_4.$
\smallskip

Let $\calh_Z$ be the pullback by $\sigma$ of a general hyperplane and set $F:=\sigma^{-1}(C)$. A straightforward computation as in the beginning of \S\,\ref{sec2} shows $(\calh_Z-F)^3=1$ from which it follows $\tau$, and then $\pi$, is birational. Moreover, $\pi$ contracts all lines in $X$ intersecting $C$. We may realize such a line, say $L\subset S\subset X$, as follows:

(a') If $C$ is as in (a), then $L$ is either $f^{-1}(p_i)$ or the strict transform of a line passing through $p_j,p_k$, $i\not\in\{j,k\}$.

(b')   If $C$ is as in (b), then $L$ is $f^{-1}(p_j)$ for a $j\not\in\{1,2,3,4\}$.\smallskip

In both cases above we deduce $\pi$ contracts four $(-1)$-curves on each general surface $S$ containing $C$, these lines are the unique ones  intersecting $C$. Furthermore, we have 
\[K_Z\cdot \sigma_*^{-1}(L)=(\sigma^*K_X+F)\cdot \sigma_*^{-1}(L)=-2\]
from which we conclude $\tau$ is an extraction and $\tau\circ\sigma^{-1}$ is a link of type II inverse of that in Examples \ref{exas1}, part $(L.2)$; in other words, $\tau=p$ and $\sigma=q$ relatively to notations therein.\smallskip

Now we describe the family $\calc_2=\calc_2(X)$ of all irreducible conics on $X$, the centers of all possible contractions $q=\sigma:Z\to X$ as above. As we saw before, a general hyperplane section of $X$ containing a fixed conic $C\in\calc_2$ is a (smooth) del Pezzo surface of degree 4 relative to that hyperplane.

\begin{lem}
The family $\calc_2$ is a pure dimension 4 quasi-projective variety such that $\caln_C X\simeq\calo_{\PP^1}(1)\oplus\calo_{\PP^1}(1)$ for all $C\in\calc_2$. Moreover, if $S\subset X$ is a smooth hyperplane section containing a conic $C\in\calc_2$, $f:S\to\plane$ the blow-up of four points, then $f(C)$ is as in $(a)$ or $(b)$ above.
\label{lem3}
\end{lem}

\begin{proof}
The quasi-projective structure on $\calc_2$ follows from the general theory on Hilbert schemes. The assertion about the dimension of $\calc$ follows from the second part of the Lemma.

\smallskip 
 
Write $C\sim aL-\sum_{i=1}^4 b_i E_i$, $b_i\geq 0$, as a divisor class in $\pic(S)$, where $L$ is the pullback of a general line in $\PP^2$ and $E_i's$ the exceptional divisors; hence $i\neq j$ implies $b_i+b_j\leq a$.

On the other hand, since $K_S\cdot_{S} C=-2$, by the adjunction formula on $S$ we get $C\cdot_S C=0$, where the subindex ``$S$'' in the dot means the intersection number is taken on $S$.  We deduce
\[3a-\sum_{i=1}^4 b_i =2,\ \ a^2-\sum_{i=1}^4 b_i^2=0.\]
Hence $a\leq 2$ from which it follows $f(C)$ is as in (a) or (b). 

Finally, let us prove the assertion about the normal bundle of $C$ in $X$. By relating normal bundles on $C$, $S$ and $X$ we obtain an exact sequence of bundles on $C$:
\begin{equation}
\xymatrix{0\ar@{->}[r]&\caln_C S\ar@{->}[r]&\caln_C X\ar@{->}[r]&\caln_S X|_C\ar@{->}[r]&0}.\label{eq-coho}
\end{equation}
Since $\caln_C S\simeq\calo_{\PP^1}$ and $\caln_S X|_C=\calo_{X}(1)\otimes\calo_C\simeq\calo_{\PP^1}(2)$, we deduce  $\caln_C X\simeq\calo_{\PP^1}(r)\oplus\calo_{\PP^1}(s)$ for integer numbers $r,s$ such that $r+s=2$, $r\leq s$. From the cohomology long sequence associated to (\ref{eq-coho}) it follows $(r,s)\in\{(-1,3), (1,1),(0,2)\}$. .  We deduce
\[3a-\sum_{i=1}^4 b_i =2,\ \ a^2-\sum_{i=1}^4 b_i^2=0.\]
Hence $a\leq 2$ from which it follows $f(C)$ is as in (a) or (b). 

Finally, let us prove the assertion about the normal bundle of $C$ in $X$. By relating normal bundles on $C$, $S$ and $X$ we obtain an exact sequence of bundles on $C$:
\begin{equation}
\xymatrix{0\ar@{->}[r]&\caln_C S\ar@{->}[r]&\caln_C X\ar@{->}[r]&\caln_S X|_C\ar@{->}[r]&0}.\label{eq-coho}
\end{equation}
Since $\caln_C S\simeq\calo_{\PP^1}$ and $\caln_S X|_C=\calo_{X}(1)\otimes\calo_C\simeq\calo_{\PP^1}(2)$, we deduce  $\caln_C X\simeq\calo_{\PP^1}(r)\oplus\calo_{\PP^1}(s)$ for integer numbers $r,s$ such that $r+s=2$, $r\leq s$. From the cohomology long sequence associated to (\ref{eq-coho}) it follows $(r,s)\in\{(-1,3), (1,1),(0,2)\}$. 

On the other hand,  the contraction $p=\tau:Z\to\PP^3$ induces a birational morphism from  the exceptional divisor $F$ of $q$ onto a smooth quadric (Remark \ref{rem2}). Since $(r,s)=(-1,3)$ or $(r,s)=(0,2)$ imply $F$ is a Nagata-Hirzebruch surface $\hir_2$, in which case that morphism can not exist, we conclude $(r,s)=(1,1)$. 
\end{proof}

Finally, we have 
\begin{lem}

Let $C,C'\in\calc_2$ be two distinct conics. Then $C$ and $C'$ are not coplanar, intersect transversally and $|C\cap C'|\in \{0,1\}$. Moreover, the two possibilities for $|C\cap C'|$ occur.
\label{lem4}
\end{lem}

\begin{proof}
Let $P$ and $P'$ be the planes where  $C$ and $C'$ lie. We know $P\cap X=C$, $P'\cap X=C'$. Hence
\[P\cap C'=P'\cap C=C\cap C',\]
and $|C\cap C'|\leq 2$. 

Suppose $C$ and $C'$ intersect at two points (counting multiplicities) and consider $\sigma:Z\to X$, the blow-up of $X$ along $C$, $F:=\sigma^{-1}(C)$ and let $\calh_Z$ be the pullback of a general hyperplane section of $X$; $\tau:Z\to \PP^3$ is the related extraction which is the blow-up of a rational quartic curve $\Gamma\subset\PP^3$; as usual $E:=\tau^{-1}(\Gamma)$. If $D$ denotes the strict transform of $C'$ by $\sigma$, then 
\[K_Z\cdot D=-2,\ \ (\calh_Z-F)\cdot D=0.\]
The second equality implies $\tau$ contracts $D$ and the first one implies $D\equiv 2E_y$ where $E_y=\tau^{-1}(y)$ for a point $y\in \Gamma$. This is not possible because $\tau|_{E}:E\to \Gamma$ is a projective fibration and $D$ is irreducible. Then  $|C\cap C'|\leq 1$.

Finally, on a (degree 5) del Pezzo surface $S\subset\PP^5$ there are pairs of disjoint conics and pairs of conics intersect each other. 
\end{proof}

\subsubsection{Case $(L.4)$}\label{subsubsec3.1.3} Here $X$ is also a smooth hyperquadric in $\PP^4$. We keep notations as in Examples \ref{exas1}. As we saw $C_5:=q(F)$ is a smooth quintic curve of genus 1, not contained in a dimension 3 linear space, and $F\in |5H-2E|$.  By construction, the link $\chi:\PP^3\tor X$ maps a general plane $\Pi\subset \PP^3$ onto the blow-up of $\Pi$ along the set of five points $\Pi\cap \Gamma=\{p_1,\ldots,p_5\}$, which is a del Pezzo surface of degree 4. Therefore $C_5$ is the image of a quintic curve in $\Pi$ with double points at $p_1,\ldots,p_5$, that is $C_5=\chi_*(p(F)\cap\Pi)$, and the inverse link $\chi^{-1}:X\tor\PP^3$ is defined by the linear system $|2\calh_X-C_5|:=q_*|2\calh_Z-F|$, a linear system of del Pezzo quartic surfaces containing $C_5$. Notice that for a general line $\ell\subset\PP^3$, the strict transform $\chi_*(\ell)$, which one may suppose to be contained in $\Pi$, is a twisted cubic 5-secant to $C_5$.

 Consider a degree 4 del Pezzo surface $S\subset\PP^4$, let $f:S\to\PP^2$ be the  blow-up of $\PP^2$ at $p_1,\ldots, p_5$, and let $C\subset S$ be a nondegenerate elliptic quintic curve; denote by $E_i:=f^{-1}(p_i)\subset S$ the corresponding exceptional divisors, $i=1,\ldots,5$. 

Suppose for a moment $C\cdot_S E_1\geq 3$. By projecting $X$ from the line $E_1$ to $\PP^3$ we obtain that $C$ is birationally equivalent to a conic: impossible. Then $E_i\cdot_S C\leq 2$ for all $i$. Write $C\sim aL-\sum_{i=1}^5 b_i E_i$, $b_i\geq 0$, as a divisor class in $\pic(S)$, where $L$ is the pullback of a general line in $\PP^2$; hence $b_i\leq 2$.

On the other hand,  since $K_S\cdot_{S} C=-5$, by the adjunction formula on $S$ we get $C\cdot_S C=5$. We deduce
\[3a-\sum_{i=1}^5 b_i =5,\ \ a^2-\sum_{i=1}^5 b_i^2=5.\]
Since $C$ is not rational (then $a\geq 3$) we conclude $3\leq a\leq 5$, and $C$ is the strict transform, under $f$, of one of the following:

(a)  a quintic curve with double points at $p_1,\ldots,p_5$ 

(b) a quartic curve with double points at $p_i,p_j$ and passing through all $p_k$ with $k\neq i,j$.

(c) a smooth cubic curve passing through four of the points $p_1,\ldots,p_5$.

A first consequence of the description above is that every elliptic quintic curve $C$ in $X\subset\PP^4$ which is not contained in a 3-space admits (five) chords: indeed, in the case (a) these chords are the (-1)-curves $f^{-1}(p_i)'s$, in case (b) these are  $f^{-1}(p_i),f^{-1}(p_j)$ and the strict transform of lines passing through two of the remaining points, etc. 

Second, such a quintic curve is the base locus scheme of a special link, inverse to a link as in case $(L.4)$: indeed, denote by $\phi:X\tor\P^s$ the rational map defined by the linear system $|2\calh_X-C|$ and by $\sigma:Z\to X$ the blow-up of $X$ along $C$.  By restricting $\phi$  to a general dimension 3 linear space intersecting $C$ transversely, we deduce $s=3$ and $\tau:=\phi\circ\sigma$ is a morphism (the linear system of quadrics in $\PP^3$ passing through five points in general position has dimension 4, contains smooth members and its indeterminacy may be resolved by blowing up these points). In particular, a  general member in $|2\calh_X-C|$ is smooth.

Let $L\subset X$ be a chord of $C$. Its strict transform $\sigma_*^{-1}(L)$ is contracted by $\tau$ and satisfies 
\[K_Z\cdot \sigma_*^{-1}(L)=(-3\calh_X+F)\cdot \sigma_*^{-1}(L)=-1;\]
that is, $\tau$ is an extraction. By Theorem \ref{thm1} it suffices to show that the center of $\tau$ is a curve. 

Let $E$ be the exceptional divisor of $\tau$ (recall that $\tau$ is the blow-up of its center). By writing $K_Z=\tau^*K_{\PP^3}+aE=-4H+aE$, for $a>0$, we deduce $H\cdot a^2E^2=-5$ from which it follows $a=1$ and the center of $\tau$ is a quintic curve.

Now we describe the family $\calc_5=\calc_5(X)$ of genus 1 smooth quintic curves on $X$ which are not contained in a three dimensional linear space.

\begin{lem}\label{lem5}
The family $\calc_5$ is a pure dimension 15 quasi-projective variety such that $\caln_C X$ is indecomposable, extension of a bundle of degree $5$ by a bundle of degree 10, for all $C\in\calc_5$.
Moreover, if $S\subset X$ is a degree 4 del Pezzo surface containing $C$, $f:S\to\PP^2$ the blow-up of five points, then $f(C)$ is as in (a), (b) or (c) above.
\end{lem}
\begin{proof}
As before, the quasi-projective structure on $\calc_5$ follows from the general theory on Hilbert schemes. Notice that all linear system of curves as in (a), (b) or (c) above has dimension 5.

Fix  $C\in\calc_5$. Since the set of hyperquadric sections of $X$ is a linear system of dimension 13 on $X$, whose elements containing $C$ define a linear subsystem of dimension 3. We deduce the assertion relative to the dimension.  

On the other hand, consider the exact sequence of bundles on $C$:
\[\xymatrix{0\ar@{->}[r]&\caln_C S\ar@{->}[r]&\caln_C X\ar@{->}[r]&\caln_S X|_C\ar@{->}[r]&0}.\]
Since $\deg(\caln_C S)=C\cdot_S C=5$ and $K_S\cdot_S C=-5$, then  $\deg(\caln_S X|_C)=10$. Moreover, since $S$ is a complete intersection, the sequence above splits if and only if $C$ is the complete intersection of $S$ with a hypersurface $V\subset\PP^4$ (\cite{HaHu}): this is not possible because $C$ is not degenerated in $\PP^4$ and $\deg(C)=5$.  We deduce that $\caln_C X$ is indecomposable which completes the proof. 

\end{proof}

\begin{cor}\label{cor5}
Let $C\in \calc_5$. There are no 3-secant lines to $C$ and $X$ contains a dimension 1 family of 2-secant lines to $C$.
\end{cor}

\begin{proof}

By projecting $C$ from one of its points we get a quartic curve with geometric genus 1, then a smooth quartic. Hence $C$ does not admit 3-secant lines. On the other hand, by using that $f(C)$ satisfies (a), (b) or (c), we conclude there are $(-1)$ curves in $S$ which are 2-secant to $C$.
\end{proof}

Finally, we have:

\begin{cor}\label{cor6}
Let $C,C'\in\calc_5$ be two genus 1 quintic curves in $X$. Then $C$ and $C'$ may be contained or not in the same hyperquadric section of $X$. Moreover, we have:

(i)  in the first case, $|C\cap C'|\leq 7$ and it may take the values $5, 6$ and $7$. 

(ii) in the second case,  $|C\cap C'|\leq 10$ and it may take all values between $0$ and $5$.
\end{cor}

\begin{proof}
The first assertion follows directly from  Lemma \ref{lem5} and its proof. 

To prove (i) we only need to consider the different possibilities to intersect curves as in (a), (b) or (c).  

To prove (ii), suppose $C$ is contained in a (degree 4) del Pezzo surface $S$ and let $Q\subset\PP^4$ be a hyperquadric such that $S':=Q\cap X$ is a del Pezzo surface distinct from $S$. We distinguish two cases:
\smallskip

\noindent (1) $S'$ does not contain $C$. A linear system on $S'$ constituted by strict transform of curves as in (a), (b) or (c) defines an immersion $\eta:S'\to \PP^5$. Since $S'$ intersects $C$ in at most $10$ points, then $ \eta(C\cap S')$ consists of at most $10$ points, being the maximum if and only if $S'$ is transversal to $C$. It suffices to chose a hyperplane in $\PP^5$ passing through $0,1,2,3,4$ or $5$ of these points, we obtain a curve $C'$ as required. 
\smallskip

\noindent(2)  $S'$ contains $C$. Then $S'\cap S$ is a degree 8 curve in $X\subset\PP^4$, one of its component is $C$. Then $S'\cap S=C\cup D$ where $D$ is curve of degree 3 (taking into account multiplicities) and arithmetic genus 0. As in the beginning of \S\ \ref{subsubsec3.1.3} we deduce $D$ intersects $C$ along a zero scheme of length 5; in particular, a quintic $C'\subset S'$, different from $C$, may intersect $C$ in at most 5 points.  
\end{proof}

\subsection{Pure Special type II Cremona transformations}

A Cremona transformation of $\PP^3$ is a birational map $\varphi:\PP^3\tor\spa$. When $\varphi$ and its inverse are defined by homogeneous polynomials of degrees $d$ and $e$, respectively, we say $\varphi$ has \emph{bidegree} $(d,e)$. In this case, if $S,S'$ are general surfaces of degree $d$ in the linear system defining $\varphi$, then $S\cap S'=\Gamma_1\cup C$ where $\Gamma_1$ contains the \emph{base locus scheme} $\bas(\varphi)$ of $\varphi$ and $C$ is a rational irreducible curve of degree $e$. Note that the theoretical-scheme structure of $\Gamma_1$ depends on $S$ and $S'$, but $\deg\Gamma_1=d^2-e$, then $\Gamma_1$ defines a unique class in the Chow group of $\PP^3$.  

We introduce the \emph{1-cycle class} associated to a Cremona transformation $\phi:\PP^3\tor\PP^3$ as the class of 1-cycles defined as above; we denote it by  $\cyc(\phi)$. 

A transformation with bidegree $(d,e)\in\{(2,2),(3,3)\}$ is said to be \emph{general} if its base locus scheme is smooth (we could have called it ``special'', but the base locus may be disconnected which disagrees with what one finds in the literature). Finally, on one side, a bidegree $(3,3)$ transformation $\varphi$ is said to be \emph{de Jonqui\`eres} if, up to linear change of coordinates in the domain and the target, we may write
\begin{equation}\label{eq:jonq}
\varphi=(g:kx:ky:kz),\ g,k\in\C[w,x,y,z]\ \mbox{homogeneous},
\end{equation}
where $g$ is irreducible and vanishes at $(1:0:0:0)$ with order $2$, and $k(1:0:0:0)=0$; in particular $(1:0:0:0)$ is an embbeded point in $\bas(\varphi)$ and the strict transform of a general line, under $\varphi$, is a plane cubic curve. On the other side, $\varphi$ is said to be \emph{determinantal} if it is defined by the maximal minors of a $4\times 3$ matrix of linear forms; of course, a general bidegree $(3,3)$ transformation is determinantal.

\begin{rem}\label{rem-end}
When $\varphi$ is as in (\ref{eq:jonq}), for general lines $\ell\subset\PP^3$, the plane curve $\varphi^{-1}_*(\ell)$ spans a general plane passing through $(1:0:0:0)$.  
\end{rem}

We say a link as in case $(L.i)$, see Examples \ref{exas1}, is a link in the class $(L.i)$, $1=1,2,3,4,5$. 

\begin{thm}
Let $\varphi:\PP^3\tor\PP^3$ be a Cremona transformation not defined somewhere which is a product of $\ell$ special links of type II, with $\ell\leq 2$. Then, exactly one of the following assertions hold: 

$(a)$ $\ell=1$ and $\varphi$ is a general Cremona transformation of bidegree $(3,3).$

$(b)$ $\ell=2$ and $\varphi$ is a product $\chi_2^{-1}\circ\chi_1$ where $\chi_1, \chi_2$ are links of type II in the class $(L.i)$, with $i=1,2,3,4,5$. 

$(c)$  $\ell=2$ and $\varphi$ is a product $\chi_2^{-1}\circ\chi_1$ where $\chi_1$ and $\chi_2$ are links of type II either in classes $(L.3)$ and $(L.4)$, respectively, or conversely.
\label{thm2}
\end{thm}

\begin{proof}
If $\ell=1$, we know $\varphi$ is a general cubo-cubic transformation by Theorem \ref{thm1}. In the sequel we assume $\ell=2$.\smallskip

By hypothesis $\varphi=\chi_2^{-1}\circ\chi_1$ where $\chi_i:\PP^3\tor X$ is a special link of type II onto a smooth 3-fold $X$, $i=1,2$. As we know by Remark \ref{rem1.0} each $\chi_i$ is not defined along an irreducible curve, say $\Gamma_i\subset\PP^3$, and then $ \chi_i=q_i\circ p_i^{-1}$ where $p_i:Z_i\to \PP^3$ is the blow-up of $\Gamma_i$. By construction $\chi_i$ is not trivial, that is, it is not an isomorphism. Theorem \ref{thm1} implies the result.
\end{proof}

For simplicity, a Cremona transformation as in Theorem \ref{thm2} is said to be \emph{Pure Special type II}. 

\begin{cor}\label{cor-thm2}
There exist eight classes of Pure Special type II Cremona transformations with $\ell\leq 2$. 
\end{cor}

In the theorems below we do not include an explicit description of Pure Special type II Cremona transformations containing links in the class $(L.5)$, where the situation should be clear but to describe it may be longer than interesting. Nevertheless, we note that such a Cremona transformation always has  
bidegree $(9,9)$; the reason is that for a link $\chi$ in class $(L.5)$ and a general line $\ell\subset\PP^3$,  the curve $(\chi^{-1})_*(\ell)$ is a twisted cubic curve which intersects $\bas(\chi)$ at (eight) variable points. 

\begin{thm}\label{thm3}
Let $\varphi=\chi_2^{-1}\circ\chi_1:\PP^3\tor\PP^3$ be a Pure Special type II Cremona transformation not defined somewhere. Suppose $\chi_1,\chi_2$ are both in the same class $(L.i)$, $i=1,\ldots,4$. Then, exactly one of the following assertions hold:

$(a)$ In the class $(L.1)$ case $\varphi$ is a bidegree $(3,3)$ transformation whose base locus scheme contains a smooth genus 2 quintic curve $\Gamma$ and it holds one of the following:
\begin{itemize}
  \item[i)] $\bas(\varphi)=\Gamma\cup L$ where $L$ is a 2-secant line to $\Gamma$ if and only if $\bas(\chi^{-1}_1)\cap \bas(\chi^{-1}_2)=\emptyset$; in this case $\varphi$ is determinantal. 
  \item[ii)] $\bas(\varphi)=\Gamma\cup L$ with an embbeded point where $L$ is a trisecant line to $\Gamma$ if and only if $\bas(\chi^{-1}_1)\cap \bas(\chi^{-1}_2)\neq \emptyset$; in this case $\varphi$ is de Jonqui\`eres.
\end{itemize}

$(b)$  In the class $(L.2)$ case $\varphi$ is a bidegree $(3,3)$ transformation whose base locus scheme contains a smooth rational quartic curve $\Gamma$ and it holds one of the following:
\begin{itemize}  

\item[i)] $\bas(\varphi)=\Gamma\cup D_2$ where $D_2$ is a rank 1 or 3 conic which is 4-secant to $\Gamma$ if and only if  $\bas(\chi^{-1}_1)\cap \bas(\chi^{-1}_2)=\emptyset$; in this case $\varphi$ is  determinantal.

 \item[ii)] $\bas(\varphi)=\Gamma\cup D_2$ where $D_2=L\cup L'$ is a rank 2 conic which is 4-secant to $\Gamma$, with $L$ being 3-secant, if and only if  $\bas(\chi^{-1}_1)\cap \bas(\chi^{-1}_2)\neq \emptyset$; in this case $\varphi$ is determinantal.
\end{itemize}

$(c)$  In the class $(L.3)$ case $\varphi$ is a general bidegree $(2,2)$ transformation whose base locus scheme is the union of a smooth conic and a point not lying on its plane.

$(d)$  In the class $(L.4)$ case  $\varphi$ is a bidegree $(6,6)$ Cremona transformation whose base locus scheme contains an elliptic quintic curve $\Gamma$ and another curve $D$ such that $\cyc(\varphi)=n\Gamma+D$, with $n\geq 4$, and it holds one of the following:
\begin{itemize}  

\item[(i)] the curves $\bas(\chi^{-1}_1)$ and $\bas(\chi^{-1}_2)$ intersect at $m$  (taking into account multiplicities) points, with $0\leq m\leq 10$, and $D$ is supported on the union of  an irreducible curve $C_{10-m}$ of degree $10-m$, which is birationally equivalent to $\bas(\chi^{-1}_2)$, and $m$ lines 3-secant to $\Gamma$; in this case $n=4$, $C_{10-m}$ is smooth in $\PP^3\backslash \Gamma$ and it is $(25-3m)$-secant to $\Gamma$.
 
\item[(ii)]   the curves $\bas(\chi^{-1}_1)$ and $\bas(\chi^{-1}_2)$ intersect  at $m$  (taking into account multiplicities) points, with $m\in\{0,5\}$, and $D$ is supported on the union of $m$ lines 3-secant to $\Gamma$; in this case $n=m=5$ or $n=6$ and $m=0$.

\end{itemize}
\end{thm}

\begin{proof}
Since $\varphi$ is not defined somewhere, then  $\bas(\chi^{-1}_1)\neq \bas(\chi^{-1}_2)$. The case (a) follows from \cite{AGSP} and the case (c) is essentially trivial. Let us consider the cases (b) and (d).

\smallskip

\noindent{\bf Case} (b). We recall and keep all notations from \S \ref{subsubsec3.1.2}, $\chi_1=q\circ p^{-1}$ and $\Gamma:=\bas(\chi_1)$. Then, for each $i=1,2$, the birational map $\chi^{-1}_i=\pi_i:X\tor \PP^3$ is the restriction to $X$ of a general projection from a plane, $\bas(\pi_1)=C\in\calc_2, \bas(\pi_2)=C'\in\calc_2$. 

The strict transform $S\subset X$ of a general plane by $\pi_2$ is a del Pezzo surface of degree 5 containing the conic $C'=\bas(\pi_2)$. As we saw $S$ is a hyperplane section of $X$ from which we deduce it intersects $C=\bas(\pi_1)$ at two points. Therefore the strict transform $(\chi_1)_*^{-1}(S)=(\pi_1)_*(S)$ is a cubic surface. By symmetry we get $\varphi$ has bidegree $(3,3)$. 

Furthermore, let $\ell\subset\PP^3$ be a general line. The strict transform $\gamma:=(\pi^{-1}_2)_*(\ell)$ of $\ell$ is a twisted cubic 2-secant to $C'$. By construction $\gamma\cup C'$ is a hyperplane section of $X$, general among those containing the plane $P'\supset C'$.

\smallskip 

Suppose $C\cap C'=\emptyset$. By genericity on $\ell$ , we may suppose that the linear 3-space $\langle\gamma\rangle$  spanned by $\gamma$ does not intersect the plane $P$ of $C$. We deduce $\pi_1$ restricts isomorphically on $\gamma$ and then $\pi_1(\gamma)$ is a twisted cubic curve in $\PP^3$.

Take irreducible cubic surfaces, say $W,W'\subset\PP^3$, such that $W\cap W'=  \Gamma\cup D_2\cup \pi_1(\gamma)$, where $D_2:=\pi_1(C')$ is a conic which has rank 3 (i.e. it is smooth) if $P\cap P'=\emptyset$ and has rank 1 otherwise. Since a twisted cubic curve is arithmetically Cohen-Macaulay of genus 0, by \emph{liaison} $\Gamma\cup D_2$ is an arithmetically Cohen-Macaulay curve of degree 6 and arithmetic genus 3, whose ideal is generated by 4 independent cubic forms which are the maximal minors of a $4\times 3$ matrix of linear forms (\cite[\S
3]{PS74}). Hence $\varphi$ is determinantal. By adjunction on $S$, we have that arithmetic genus 3 for $\Gamma\cup D_2$ implies $D_2$ is 4-secant to $\Gamma$.

Now suppose  $C\cap C'$ consists of a unique point, say $x\in X$; by Lemma \ref{lem4} there is no other possibility. Then $P\cap P'=\{x\}$ from which it follows $P'$ contracts (via $\pi_1$) to a line $L'$. We deduce this line is a chord of $\varphi^{-1}_*(\ell)=(\pi_1)_*(\gamma)$ and this curve is a twisted curve, which as before implies $\varphi$ is determinantal: in fact, otherwise it would be de Jonqui\`eres and $\varphi^{-1}_*(\ell)$ contradicts what we said in Remark \ref{rem-end}.

Furthermore, $\pi_2\circ q$ in not defined along the line (fiber) $F_x=q^{-1}(x)$ which corresponds via $p$ to a line intersecting $L$; that line is 3-secant to $\Gamma$ since $H=\calh_Z-F$ and $\calh_Z=3H-E$, $\calh_Z\in |q^*\calo_X(1)|$,  implies $E=2\calh_Z-3F$, then $E\cdot F_x=3$.

If $D_2=L\cup L'$, then $D_2$ is 4-secant to $\Gamma$, as before, 
which completes the proof of the statement (b).

\smallskip

\noindent{\bf Case} (d). Now $C,C'\in\calc_5$ are elliptic quintic curves on the hyperquadric $X\subset \PP^4$. The reduced structure on $\bas(\varphi)$ is then $\Gamma\cup p(q^{-1}(C'))$, where $q\circ p^{-1}=\chi_1$. The (classes of)  exceptional divisors $E$ and $F$ of $p$ and $q$, respectively, satisfy
\begin{equation}
\calh_Z:=q^*(\calh_X)=3H-E, F=5H-2E,\label{eq3.2-1}
\end{equation}
or equivalently
\begin{equation}
H= 2\calh_Z-F, E= 5\calh_Z-3F.\label{eq3.2-2}
\end{equation}
Then $\varphi$ is defined by a linear subsystem of $p_*q^*(|2\calh_X|)=p_*(|6H-2E|)$, which consists of degree 6 surfaces passing through $\Gamma$ with multiplicity at least $2$; in particular, $\cyc(\varphi)=n\Gamma+D$ with $n\geq 4$ and, by symmetry, $\bideg(\varphi)=(6,6)$. 

On the other hand,  $m=q^{-1}_*(C')\cdot F$ by definition.

Suppose $C'$ is not contained in $q(E)$, that is $q^{-1}_*(C')\not\subset E$. The restriction of $p$ to $q^{-1}_*(C')$ induces a birational morphism which is injective on $Z\backslash E$. Then $C_{10-m}:=p(q^{-1}_*(C'))$ is a curve of degree $10-m$, smooth on $\PP^3\backslash \Gamma$, which intersects $\Gamma$ at a zero scheme of length $s=q^{-1}_*(C')\cdot E$. We deduce 
\[
s=q^{-1}_*(C')\cdot (5\calh_Z-3F)= 25-3m.\]
Take general elements  $S,S'$ in the linear system defining $\varphi$. The scheme theoretical intersection $S\cap S'$ defines a 1-cycle of degree $36$ supported on
\[\Gamma\cup p(q^{-1}_*(C'))\cup p(q^{-1}(C\cap C').\]
By taking into account multiplicities, the last part in the union above consists of  $m$ lines 3-secant to $\Gamma$. By computing degrees we deduce that $\cyc(\varphi)$ has degree
\[36-6=n\deg\Gamma+10,\]
that is, $n=4$.

\smallskip 

Now suppose $C'\subset q(E)$. Since $E$ is a projective lines bundle over $\Gamma$ we deduce $p(q_*^{-1}(C'))=\Gamma$, then $n>4$. The last assertions follow easily.  
\end{proof}

\begin{rem}\label{rem-table1}
\noindent(1) Transformations as in (a)i  correspond to type $T_{33}^{(3)}$  in table \cite[pag. 185]{SR}. Analogously, transformations as in (b)i correspond to type $T_{33}^{(4)}$ in \cite[pag. 185]{SR} together with degenerations of this type.

\noindent(2) We do not know if the case (ii) in part (d) of the theorem above may effectively occur. 
\end{rem}

\begin{thm}\label{thm4}
Let $\varphi=\chi_2^{-1}\circ\chi_1:\PP^3\tor\PP^3$ be a Pure Special type II  Cremona transformation. Suppose $\chi_1$  and $\chi_2$ are links of type II either in classes  $(L.3)$ and  $(L.4)$, respectively, or conversely.  Then, exactly one of the following assertions hold: 

$(a)$ In the first case the base locus scheme of $\varphi$ contains a smooth conic $\Gamma$ and it holds one of the following:
\begin{itemize}
  \item[(i)] $\bas(\varphi)_{red}=\Gamma\cup D_5$ where $D_5$ is a genus 1 quintic curve, with at most a double point, which is 5-secant to $\Gamma$ if and only if $\bas(\chi^{-1}_1)\cap \bas(\chi^{-1}_2)=\emptyset$; in this case $\bideg(\varphi)=(4,3)$ and $\cyc(\varphi)=4\Gamma+D_5$.
  \item[(ii)] $\bas(\varphi)_{red}=\Gamma\cup D_4$  where $D_4$ is an elliptic quartic curve 3-secant to $\Gamma$ if and only if $\bas(\chi^{-1}_1)\cap \bas(\chi^{-1}_2)\neq \emptyset$; in this case $\bideg(\varphi)=(3,3)$ and $\cyc(\varphi)=\Gamma+D_4$.
\end{itemize}

$(b)$ In the second case the base locus scheme of $\varphi$ contains an elliptic quintic curve $\Gamma$ and  it holds one of the following:  

\begin{itemize}
  \item[(i)] $\bas(\varphi)=\Gamma\cup \{o'\}$ where $o'$ is either an infinitely near point or an isolated point if and only if $\bas(\chi^{-1}_1)\cap \bas(\chi^{-1}_2)=\emptyset$; in this case $\bideg(\varphi)=(3,4)$ and $\cyc(\varphi)=\Gamma$.
\item[(ii)] $\bas(\varphi)=\Gamma\cup L$  where $L$ is a line 3-secant to $\Gamma$ if and only if $\bas(\chi^{-1}_1)\cap \bas(\chi^{-1}_2)\neq \emptyset$; in this case $\bideg(\varphi)=(3,3)$ and $\cyc(\varphi)=\Gamma+L$.
\end{itemize}
\end{thm}

\begin{proof}
Recall the description given in \S\ \ref{subsubsec3.1.3}. The linear system difining a special link of class $(L.4)$ consists of hyperquadric sections of $X$ containing a quintic elliptic curve $C_5\in\calc_5$; we denote by $|2\calh_X-C_5|$ such a linear system. On the other hand, a special link of class $(L.3)$ is the projection of $X\subset\PP^4$ from a point $o\in X$.

We keep notations as in the previous theorem. The divisors $E,F,H,\calh_Z$ always refer to the morphisms $p$ and $q$ where $\chi_1=q\circ p^{-1}$.

\smallskip

{\noindent}(a) In this case we have
\[H=\calh_Z-F,\ \calh_Z=2H-E\]
where $q(F)=\{o\}$ and $p(E)=\Gamma$.

First suppose $o\not\in C_5$. Then $C_5$ projects birationally to a quintic curve $D_5$. Notice that $D_5$ has at most one double point, and this occurs only when $o$ belongs to a 2-secant line to $C_5$ (see Corollary \ref{cor5}). This curve intersects $\Gamma$ in a $0$-scheme of length $s:=q_*^{-1}(C_5)\cdot E$. Since $E=\calh_Z-2F$ and $F\cap q_*^{-1}(C_5)=\emptyset$, we get $s=5$. Moreover, $p_*q^*|2\calh_X-C_5|\subset p_*|2\calh_Z|=|4H-2E|$, from which we deduce that $\varphi$ is defined by a linear system of quartic surfaces containing $\Gamma$ with multiplicity $\geq 2$; in particular $\deg(\varphi)=4$ and $\cyc(\varphi)=n\Gamma+D_5$ with $n\geq 4$. 

On the other hand, since a twisted cubic on $X$ which does not pass through $o$ projects to a cubic curve we obtain $\deg(\varphi^{-1})=3$, hence $\bideg(\varphi)=(4,3)$. 

Finally, the 1-cycle above has degree $16-3=2n+5$, then $n=4$, which completes the proof of (i).

Now suppose $o\in C_5$; in this case $C_5$ projects to an elliptic quartic curve $D_4$. Hence $F\cdot q_*^{-1}(C_5)=1$ and $q^{-1}(C_5)=q_*^{-1}(C_5)\cup F$. So $D_4\cdot E$=$D_4\cdot (\calh_Z-2F)=4$. We deduce that $D_4$ is an elliptic quartic curve 3-secant to $\Gamma$. The rest of the statement (ii) follows by arguing as before.

\smallskip

\noindent(b) The assertion relative to bidegrees follows from (a). 

To prove the remaining part, first suppose $o\not\in C_5$. If $o\in q(E)$, then $o':=q^{-1}(o)\in E$ and $p(o')\in\Gamma$. For a general line $\ell\subset\P^3$, the strict transform $(\chi_2^{-1}\circ q)^{-1}_*(\ell)$ is a curve passing through $o'$, then $\varphi\circ p$ in not defined at $o'$; we deduce $o'$ is an infinitely near point lying over $\Gamma$.  

If $o\not \in q(E)$, then  $o':=p(q^{-1}(o))\not\in\Gamma$ defines an isolated point of $\bas(\varphi)$.

Finally, we suppose $o\in C_5$. In this case the fiber $F_o:=q^{-1}(o)\subset F$ is a smooth rational curve. From (\ref{eq3.2-2}) we get 
\[H\cdot F_o=1,\ \ E\cdot F_o=3.\]
Then $p(F_o)$ is a line 3-secant to $\Gamma$.
\end{proof}

\begin{rem}\label{rem-table2}
With a small additional effort we may prove that the bidegree $(3,3)$ Cremona transformations appearing in Theorem \ref{thm4} are determinantal. In fact, in cases (a)ii and (b)ii, the base locus $\bas(\varphi)$ is a sextic curve of arithmetic genus 3 and one may prove that such a curve is arithmetically Cohen-Macaulay (or use \cite{Pa97}). The birational maps in (a)ii and (b)ii correspond, respectively, to the cases $T^{(6)}_{33}$ and $T^{(2)}_{33}$ in table \cite[pag. 185]{SR}.
\end{rem}

From the theorems above it follows (to compare with Remarks \ref{rem-table1} and \ref{rem-table2}):

\begin{cor}\label{corfinal}
The Cremona transformations of types $T_{33}^{(1)}, T_{33}^{(5)}, T_{33}^{(7)}$ and $ T_{33}^{(8)}$ in table \cite[pag. 185]{SR} are not Pure Special type II. 
\end{cor}

\end{document}